\documentclass[12pt,reqno]{amsart}

\usepackage[margin=1.25in]{geometry}
\usepackage[utf8]{inputenc}
\usepackage[T1]{fontenc}
\usepackage{amsmath,amssymb,amsfonts,mathtools}
\usepackage{enumitem}
\usepackage[colorlinks=true,linkcolor=blue,citecolor=blue,urlcolor=blue]{hyperref}
\usepackage[nameinlink,noabbrev]{cleveref}
\usepackage{bbm}

\newtheorem{thm}{Theorem}[section]
\newtheorem{lem}[thm]{Lemma}
\newtheorem{cor}[thm]{Corollary}
\newtheorem{prop}[thm]{Proposition}

\theoremstyle{definition}

\theoremstyle{remark}

\crefname{thm}{Theorem}{Theorems}
\crefname{lem}{Lemma}{Lemmas}
\crefname{cor}{Corollary}{Corollaries}
\crefname{prop}{Proposition}{Propositions}
\crefname{claim}{Claim}{Claims}
\crefname{defn}{Definition}{Definitions}
\crefname{notation}{Notation}{Notations}
\crefname{remark}{Remark}{Remarks}

\numberwithin{equation}{section}

\newcommand{\R}{\mathbb{R}}
\newcommand{\Sph}{\mathbb{S}}
\newcommand{\Sym}{\mathcal{S}}
\newcommand{\diver}{\operatorname{div}}

\newcommand{\eps}{\varepsilon}
\newcommand{\res}{\mathop{\hbox{\vrule height 7pt width .5pt depth 0pt
			\vrule height .5pt width 6pt depth 0pt}}\nolimits}

\title{Bellman Equations with Sub-Lipschitz Hessians}
\author{Xavier Fern\'andez-Real}
\address{EPFL SB, Station 8, 1015 Lausanne, Switzerland}
\email{xavier.fernandez-real@epfl.ch}
\thanks{The author was supported by the Swiss National Science Foundation (SNF grant PZ00P2\_208930), by the Swiss State Secretariat for Education, Research and Innovation (SERI) under contract number MB22.00034, and by the AEI project PID2024-156429NB-I00 (Spain).}
\date{}


\begin{document}

\begin{abstract}
We construct homogeneous solutions with non-Lipschitz Hessian for finite, constant-coefficient Bellman equations.  First, for every $\sigma\in(0,1)$, we find two uniformly elliptic matrices $A_1,A_2\in\Sym^4$ and a nonzero $(2+\sigma)$-homogeneous solution $u$ of
\[\max\bigl\{{\rm tr}\,(A_1D^2u),{\rm tr}\,(A_2D^2u)\bigr\}=0 \qquad\text{in }\R^4.\]
 Second, in $\R^2$ we construct three matrices satisfying ${\rm Id}_2\leq A_j\leq3{\rm Id}_2$ for which the corresponding Bellman equation admits a homogeneous solution with a non-Lipschitz Hessian.

 In particular, solutions to convex fully nonlinear uniformly elliptic equations are not in $C^{2,1}$, and not even in $C^{2, 1-\eps}$ for $\eps > 0$ small.
\end{abstract}

\maketitle

\section{Introduction}\label{sec:introduction}

\subsection{The regularity problem}\label{subsec:regularity-problem}

One of the central questions in elliptic PDE is how much regularity is forced by the equation itself.  Consider
\[F(D^2u)=0 \qquad\text{in }B_1\subset\R^n,\]
where $\Sym^n$ denotes the space of real symmetric $n\times n$ matrices and $F:\Sym^n\to\R$ is uniformly elliptic.  Thus, for some $0<\lambda\leq\Lambda$,
\[\lambda{\rm tr}\,(N)\leq F(M+N)-F(M)\leq\Lambda{\rm tr}\,(N)\]
whenever $M\in\Sym^n$ and $N\geq0$.

When $F$ is convex or concave, the Evans--Krylov theorem gives
\[u\in C^{2,\alpha_0}_{\rm loc}, \qquad \alpha_0=\alpha_0(n,\lambda,\Lambda)>0;\]
see \cite{Evans1982,Krylov1983,CaffarelliCabre1995,FernandezRealRosOton2022}.

The regularity of the operator suggests a natural scale beyond this estimate.  If, in addition, $F\in C^{1,\gamma}$ for some $\gamma\in(0,1)$, then Evans--Krylov followed by Schauder theory gives
$u\in C^{3,\gamma}_{\rm loc}$;
see \cite{FernandezRealRosOton2022}.  Still assuming that $F$ is convex or concave, if $F\in C^1$, then combining Evans--Krylov with the $W^{2,p}$ estimates for the equations satisfied by incremental quotients yields
\[u\in W^{3,p}_{\rm loc}\quad\text{for every }1<p<\infty,
\qquad\text{hence}\qquad
u\in C^{2,\beta}_{\rm loc}\quad\text{for every }0<\beta<1;
\]
see \cite{FernandezRealRosOton2022,GilbargTrudinger2001}.  The constants in this $C^1$ conclusion depend on the modulus of continuity of $DF$ and deteriorate as $\beta\uparrow1$; thus they do not pass uniformly to the Lipschitz class.  Since uniform ellipticity already forces $F$ to be Lipschitz, this leads to a natural endpoint question for convex operators: does Lipschitz regularity of $F$ force $u\in C^{2,1}_{\rm loc}$, or at least $u\in C^{2,1-\varepsilon}_{\rm loc}$ for every $\varepsilon\in(0,1)$?

A basic class of nonsmooth operators is given by Bellman equations,
\[F(M)=\max_{a\in\mathcal A}{\rm tr}\,(A_aM), \qquad \lambda {\rm Id}_n\leq A_a\leq\Lambda {\rm Id}_n.\]
They arise naturally in dynamic programming and stochastic control; see, for example, \cite{Brezis1978,BrezisEvans1979} and \cite[Appendix~C]{FernandezRealRosOton2022}.  We focus here on a finite family of constant matrices.  Already for two operators the equation reads
\[\max\{L_1u,L_2u\}=0, \qquad L_i u={\rm tr}\,(A_iD^2u).\]
It is linear in each region where one of the two operators is active, but the interface where the active operator changes is determined by the solution. Thus, this apparently elementary equation contains a transmission or free boundary structure.

The two-operator equation was studied through variational inequalities by Br\'ezis and Br\'ezis--Evans \cite{Brezis1978,BrezisEvans1979}.  They asked whether smooth data force $W^{3,p}$ regularity for every finite $p$, or $C^{2,\alpha_0}$ regularity for every $0<\alpha_0<1$, and noted that $C^3$ regularity may fail (see \cite[Open problem 1]{Brezis1978} or \cite[Remark 4.2]{BrezisEvans1979}).

In the plane, Caffarelli, De Silva, and Savin settled the endpoint question in the positive for two constant-coefficient operators.  They proved
\[\|u\|_{C^{2,1}(B_{1/2})}\leq C\|u\|_{L^\infty(B_1)}, \qquad C=C(\lambda,\Lambda),\]
for the maximum of two operators in $\R^2$ \cite[Theorem~1.2]{CaffarelliDeSilvaSavin2018}.  Thus, the endpoint expectation is valid in this particular elliptic case.

For comparison, the analogous endpoint can fail at the parabolic level.  Caffarelli and Stefanelli constructed locally bounded solutions of
$u_t=\max\left\{\Delta u,\frac12\Delta u\right\}$
using radial self-similar profiles, for which $u_t$ is not Lipschitz (so solutions fail to be $C^{2,1}$ in the parabolic scale) \cite{CaffarelliStefanelli2008}.  Since the spatial operators are proportional, this phenomenon is intrinsically parabolic and does not answer the elliptic question.

In higher dimensions, a modern formulation of the elliptic problem appears in \cite[Section~4.5]{FernandezRealRosOton2022}:
\[\text{For} \  \max\{L_1u,L_2u\}=0 \  \text{in }B_1\subset\R^n,\ n\geq3,\,  \text{must every solution belong to }C^{2,1}_{\rm loc}?\]
More broadly, the optimal regularity of nonsmooth Bellman equations (or general convex fully nonlinear uniformly elliptic operators) remains a central open question.  Even for a finite family of operators, the optimal exponent in the $C^{2,\alpha}$ estimate is not known in general; see \cite{Savin2025}, as well as \cite{RosOtonTorresLatorreWeidner2025, Goffi2026}.

The previous discussion leaves two immediate questions for finite Bellman equations: whether the planar endpoint persists for larger families, and whether the two-operator endpoint persists in higher dimensions.  More generally, one may ask what regularity is forced by a convex Lipschitz operator.  The results below show that neither of these finite-family extensions persists: three operators already allow a non-Lipschitz Hessian in dimension 2, while in dimension 4 two operators admit solutions whose exact Hessian H\"older exponent is any prescribed $\sigma\in(0,1)$.

The use of global homogeneous solutions is natural in this context.  Such solutions arise as blow-up limits, and Liouville or classification principles for them often lead to sharp local estimates; see, for instance, \cite{Huang2002}.  Their possible homogeneities are therefore directly tied to the optimal regularity problem.  Nevertheless, even for a Lipschitz convex operator in $\R^2$, the possible homogeneities and classification of global solutions with subcubic growth remain open.  We emphasize this point in \cite[Section~1.3]{ColomboFernandezRealRosOton2025} and obtain a classification under the additional structure of a two-dimensional thin-obstacle problem.  The examples below complement this picture by showing that noninteger homogeneities in the subcubic range genuinely occur for finite Bellman operators.

In particular, we show that solutions to convex fully nonlinear uniformly elliptic equations are not necessarily in $C^{2,1}$ or even $C^{2,1-}$ in general. 

\subsection{Main results}\label{subsec:main-results}

Our first result gives a negative answer to the two-operator question in dimension four.  In fact, every exponent between two and three occurs as the homogeneity of a global solution.

\begin{thm}\label{thm:main}
For every $\sigma\in(0,1)$, there are $A_1,A_2\in\mathcal{S}^4$ uniformly elliptic and a nonzero function $u\in C^{2,\sigma}_{\mathrm{loc}}(\R^4)$, homogeneous of degree $2+\sigma$, such that
\[\max\bigl\{{\rm tr}\,(A_1D^2u),{\rm tr}\,(A_2D^2u)\bigr\}=0 \qquad\text{in }\R^4.\]
In particular, $u\notin C^{2,\beta}$ at the origin for every $\beta>\sigma$.
\end{thm}

In the solution constructed by the proof, the switching set is also explicit. After identifying $\R^4=\R^2\times\R^2$ and indexing the two operators as below,
\[\Sigma(u):=\{z\in\R^4:L_1u(z)=L_2u(z)\}=\{(x,y)\in\R^2\times\R^2:|x|=|y|\}.\]
In each of the two complementary cones exactly one operator attains the maximum, whereas both operators vanish on $\Sigma(u)$.  Thus, the switching set is a nonflat cone, smooth away from an isolated conical singularity at the origin.

Hence, already for two operators in dimension four, one cannot expect $C^{2,1}$ regularity, or even any $C^{2,\beta}$ regularity for a $\Lambda/\lambda$-ratio-independent value $\beta$: any prescribed $\sigma\in(0,1)$ occurs as the exact H\"older exponent of the Hessian at the origin.  Necessarily, by the Evans--Krylov theorem, the matrices depend on $\sigma$, and their ellipticity ratio diverges as $\sigma\downarrow0$.  The corresponding two-operator question in dimension three remains open.

\medskip

Our second result stays in the plane and instead adds a third operator.  Set
\[\sigma_*:=\frac{\arctan 2}{\arctan 3}\in(0,1),\qquad\text{so that}\quad \sigma_*\approx 0.886\ldots . \]

\begin{thm}\label{thm:three-operators}
There exist matrices $A_0,A_1,A_2\in\Sym^2$ satisfying
\[{\rm Id}_2\leq A_j\leq3{\rm Id}_2, \qquad j=0,1,2,\]
and a nonzero function $u\in C^{2,\sigma_*}_{\mathrm{loc}}(\R^2)$, homogeneous of degree $2+\sigma_*$, such that
\[\max_{j=0,1,2}{\rm tr}\,(A_jD^2u)=0 \qquad\text{in }\R^2.\]
In particular, $u\notin C^{2,\beta}$ at the origin for every $\beta>\sigma_*$.
\end{thm}

The set where the active operator is not unique is exactly the $Y$-shaped profile
\[{\textstyle \bigcup}_{k=0}^2 \left\{r e^{i(2k+1)\pi/3}:r\geq0\right\}.\]
At every nonzero point of one of these rays, exactly two adjacent operators attain the value zero and the third one is strictly negative, while all three operators vanish at the origin.  Thus, the construction gives a $Y$-shaped triple junction of active Bellman regimes.

In particular, even at the fixed ellipticity ratio $\Lambda/\lambda=3$, both the planar $C^{2,1}$ endpoint and the weaker conclusion $C^{2,1-\varepsilon}$ for every $\varepsilon\in(0,1)$ fail as soon as a third operator is allowed.  Thus the number of operators is essential in the planar endpoint theorem: two operators give $C^{2,1}$ regularity, while three already allow a non-Lipschitz Hessian.  As far as we know, this is the first such construction for a finite constant-coefficient Bellman operator in the plane.

\subsection{The transmission viewpoint}\label{subsec:transmission-viewpoint}

The two-operator equation also admits a useful transmission formulation.  If $u$ solves
\[\max\{{\rm tr} (A_+ D^2u),{\rm tr} (A_- D^2u)\}=0 \qquad\text{and}\qquad q:=(L_+-L_-)u,\]
then $L_+q^+=L_-q^-,$ where $q^\pm = \max\{\pm q, 0\}$, or equivalently
\[\diver\bigl(A_+\nabla q^+-A_-\nabla q^-\bigr)=0.\]
This is a broken-conductivity equation with matrix coefficients.  Its Lipschitz theory and the structural assumptions it requires are studied in \cite{KimLeeShahgholian2017,Kim2021}.  Related sublinear cone profiles occur in anisotropic segregation: in dimensions $n\geq3$, Soave and Terracini constructed disjoint profiles of common homogeneity below one \cite{SoaveTerracini2023}.  In the symmetric constant-coefficient setting considered here, Section~\ref{sec:proofs} produces a profile of every homogeneity $\sigma\in(0,1)$, with the additional conormal compatibility needed to lift it to a Bellman solution; see Corollary~\ref{cor:broken-conductivity}.

\subsection{Ideas of the constructions}\label{subsec:ideas}

Let us first explain the four-dimensional example.  Write $\R^4=\R^2\times\R^2$.  Given $\sigma\in(0,1)$, we choose a symmetric cone whose spherical cross-section has first Dirichlet eigenvalue $\sigma(\sigma+2)$.  Its positive first eigenfunction generates a $\sigma$-homogeneous harmonic profile in the cone.

Two anisotropic copies of this profile are then placed in the complementary cones $\{|x|>|y|\}$ and $\{|y|>|x|\}$.  The two coefficient matrices are obtained from one another by exchanging the $\R^2$ factors.  This symmetry is the key point: it makes the conormal source measures of the two zero extensions agree exactly.  We thereby obtain a $\sigma$-homogeneous sign-changing function $q$ satisfying
\[L_+q^+=L_-q^- \qquad\text{in the sense of distributions}.\]

This compatibility allows us to lift $q$ by two derivatives.  A  homogeneous Poisson problem gives a $(2+\sigma)$-homogeneous function $v$ such that
\[L_-v=q^+, \qquad L_+v=q^-.\]
Since $q^+,q^-\geq0$ have disjoint supports, $\min\{L_-v,L_+v\}=0$.

The passage from a two-operator Bellman equation to a transmission problem is made explicit in \cite[Lemmas~2.1--2.2]{AnderssonMikayelyan2013} and \cite[Section~4]{CaffarelliDeSilvaSavin2018}; the lifting argument above runs this mechanism in the opposite direction.  Cone profiles of sublinear homogeneity also appear in anisotropic two-phase theory, notably \cite[Proposition~3.7]{SoaveTerracini2023}.  The broader cone and two-phase picture is related to the monotonicity methods of \cite{AltCaffarelliFriedman1984,CaffarelliJerisonKenig2002}, although no monotonicity formula is used here.  Related problems with discontinuous conductivities or free transmission were studied in \cite{AmaralTeixeira2015,KimLeeShahgholian2017,KimLeeShahgholian2019,Kim2021}.

The planar example is independent and more explicit.  We start from an anisotropically scaled complex power in a sector of opening $2\pi/3$ and extend it by rotations.  The exponent $\sigma_*$ is chosen so that the first derivatives match across the boundary rays, while a direct Hessian computation identifies the active operator in each sector.

Section~\ref{sec:proofs} contains the cone, transmission, and lifting construction and proves Theorem~\ref{thm:main}.  Section~\ref{sec:three-operators} proves Theorem~\ref{thm:three-operators}.

\section{Proof of Theorem~\ref{thm:main}}\label{sec:proofs}

\subsection{Lifting to a Bellman solution}\label{subsec:lifting}

The main result will be a consequence of the following abstract lifting statement, which converts broken quasilinear PDEs into solutions of a two-operator Bellman equation.

\begin{lem}\label{lem:lifting}
Let $\sigma\in (0, 1)$, and let
\[L_i w={\rm tr}\,(A_iD^2w), \qquad i=1,2,\]
be constant-coefficient uniformly elliptic operators on $\R^n$. Let $q_i\in C^{\sigma}_{\mathrm{loc}}(\R^n)$ be $\sigma$-homogeneous for $i = 1, 2$ and satisfying
\[L_2q_1=L_1q_2 \qquad\text{distributionally in $\R^n$}.\]
Then, there exists $v\in C^{2,\sigma}_{\mathrm{loc}}(\R^n)$, $(2+\sigma)$-homogeneous such that
\[L_1v=q_1, \qquad L_2v=q_2 \qquad\text{in }\R^n.\]
\end{lem}

\begin{proof}
Notice that if we define $v = L_1^{-1} q_1$, then $L_2 v = L_2 L_1^{-1} q_1 = L_1^{-1} L_2 q_1 = q_2$. So this is our desired candidate. Let us formalize it.

We first solve
\[L_1v=q_1\]
in the class of functions homogeneous of degree $2+\sigma$, where $q_1\in C^\sigma_{\rm loc}(\R^n)$ is $\sigma$-homogeneous (because $\sigma\in (0, 1)$). After a standard change of variables, we may assume $L_1 = \Delta$. Write
\[q_1(r\theta)=r^\sigma g(\theta)\]
and we are searching for $v$ with $\Delta v = q_1$ with
\[v(r\theta)=r^{2+\sigma}\psi(\theta).\]
The equation on the sphere is
\begin{equation}\label{eq:sphere-resolvent}\left[\Delta_{\Sph^{n-1}}+(2+\sigma)(\sigma+n)\right]\psi=g.\end{equation}
The spectrum of $-\Delta_{\Sph^{n-1}}$ is $\{\ell(\ell+n-2):\ell=0,1,2,\dots\}$, which never coincides with $(2+\sigma)(\sigma+n)$ for any integer $\ell\geq0$.  Hence the operator in \eqref{eq:sphere-resolvent} is invertible.  Standard Schauder theory (cf. \cite[Chapter~6]{GilbargTrudinger2001}) gives $\psi\in C^{2,\sigma}(\Sph^{n-1})$, and $\psi$ then extends to $v$ homogeneous of degree $2+\sigma$ with the same regularity.
Indeed, for $x\neq0$ one has $D^2v(x)=|x|^\sigma B(x/|x|)$ with $B\in C^\sigma(\Sph^{n-1})$.  Homogeneity gives $v=O(|x|^{2+\sigma})$, $Dv=O(|x|^{1+\sigma})$, and $D^2v=O(|x|^\sigma)$, so $v$ is $C^2$ at the origin with $D^2v(0)=0$.  The two cases $|x-y|\geq\frac12\max\{|x|,|y|\}$ and $|x-y|<\frac12\max\{|x|,|y|\}$ then give $|D^2v(x)-D^2v(y)|\leq C|x-y|^\sigma$; hence the extension is $C^{2,\sigma}$ also at the origin.
We have just found (after changing variables back) a $(2+\sigma)$-homogeneous $v\in C^{2,\sigma}_{\rm loc}(\R^n)$.

Define now
\[h:=L_2v-q_2.\]
We thus have, distributionally in $\R^n$,
\begin{align*}L_1h=L_1L_2v-L_1q_2=L_2L_1v-L_1q_2=L_2q_1-L_1q_2=0.\end{align*}
In particular, $h$ is smooth and $\sigma$-homogeneous.  By Liouville's theorem with growth \cite[Proposition~1.19]{FernandezRealRosOton2022}, $h$ is a polynomial of degree at most $\lfloor\sigma\rfloor=0$.  Hence $h$ is constant, and its $\sigma$-homogeneity yields $h\equiv0$, as we wanted.
\end{proof}

\subsection{A cone with prescribed characteristic exponent}\label{subsec:cone}
We now construct the homogeneous harmonic function of prescribed degree $\sigma\in(0,1)$ used in the proof.  Similar cone profiles appear in anisotropic two-phase theory; see \cite{CaffarelliDeSilvaSavin2018,SoaveTerracini2023}.  The cone constructed below has two crucial properties: its characteristic exponent can be made arbitrarily small, and it has a symmetry which helps match the two conormal source measures.

Let us denote points in $\R^4$ as
\[\R^4 \ni z = (x, y) \in \R^2\times \R^2.\]
Fix $\sigma\in(0,1)$.  For $0<\kappa\leq1$, define the cone
\[D_\kappa:=\{(x, y)\in\R^2\times\R^2:|x|>\kappa|y|\}\]
and its spherical link
\[\Omega_\kappa:=D_\kappa\cap\Sph^3.\]
Let $\mu_1(\kappa)>0$ be the first Dirichlet eigenvalue of $-\Delta_{\Sph^3}$ on $\Omega_\kappa$.  We first record its behavior as $\kappa\downarrow0$.

Use Hopf coordinates on $\Sph^3$
\[x=\omega \sin\theta,\qquad y=\eta \cos\theta,\qquad\text{where}\quad \theta\in [0, \pi/2],\ \omega, \eta \in \Sph^1,\]
so that
\[\Omega_\kappa=\{\theta>\theta_\kappa\},\qquad \theta_\kappa:=\arctan \kappa,\]
and the surface measure on $\Sph^3$ is
\begin{equation}\label{eq:hopf-metric}dS=\sin\theta\cos\theta\,d\theta\,d\omega\,d\eta.\end{equation}
We then have:

\begin{lem}\label{lem:eigenvalue}
For every $0<\kappa<1$,
\begin{equation}\label{eq:eigen-bound}\mu_1(\kappa)\leq\frac{4}{\log(1/\kappa)}.\end{equation}
In particular, $\mu_1(\kappa)\to0$ as $\kappa\downarrow0$.
\end{lem}

\begin{proof}
Consider the $H_0^1(\Omega_\kappa)$ test function
\[
\zeta_\kappa(\theta)=
\begin{cases}
\displaystyle\frac{\log(\tan\theta/\kappa)}{\log(1/\kappa)},
&\theta_\kappa\leq\theta\leq\pi/4,\\[3mm]
1,&\pi/4\leq\theta\leq\pi/2.
\end{cases}
\]
Since $\frac{d}{d\theta}\log(\tan\theta)=\frac{1}{\sin\theta\cos\theta},$
\eqref{eq:hopf-metric} gives (also using $|\nabla_{\Sph^3}\zeta_\kappa|^2 = |\zeta_\kappa'(\theta)|^2$),
\[\int_{\Omega_\kappa}|\nabla_{\Sph^3}\zeta_\kappa|^2dS=4\pi^2\int_{\theta_\kappa}^{\pi/4}\frac{1}{\log(1/\kappa)^2\sin^2\theta\cos^2\theta}\sin\theta\cos\theta\,d\theta=\frac{4\pi^2}{\log(1/\kappa)}.\]
On the other hand,
\begin{align*}
\int_{\Omega_\kappa}\zeta_\kappa^2dS
&\geq4\pi^2\int_{\pi/4}^{\pi/2}\sin\theta\cos\theta\,d\theta
=\pi^2.
\end{align*}
The Rayleigh quotient yields \eqref{eq:eigen-bound}.
\end{proof}

The map $\kappa\mapsto\mu_1(\kappa)$ is continuous and strictly increasing on $(0,1]$.  Indeed, the domains $\Omega_\kappa$ depend smoothly on $\kappa$, which gives continuity of the first eigenvalue.  If $0<\kappa_1<\kappa_2\leq1$, then $\Omega_{\kappa_2}\subsetneq\Omega_{\kappa_1}$, and strict domain monotonicity for the first Dirichlet eigenvalue gives $\mu_1(\kappa_1)<\mu_1(\kappa_2)$.

At the other endpoint, $\mu_1(1)=8$.  To see this directly, consider
\[P(x,y):=|x|^2-|y|^2,\]
which is harmonic and 2-homogeneous in $\R^4$, so in particular it has eigenvalue $2(2+2) = 8$ on the restriction to $\Sph^3$ for the operator $-\Delta_{\Sph^3}$. Moreover, $P>0$ in $\Omega_1=\{|x|>|y|\}\cap\Sph^3$ and $P=0$ on $\partial\Omega_1$, so it is the first Dirichlet eigenfunction on $\Omega_1$, and hence $\mu_1(1)=8$.

Since $0<\sigma(\sigma+2)<3<8$, the preceding lemma, continuity, and the intermediate value theorem give a unique $\kappa(\sigma)\in(0,1)$ such that
\[\mu_1\bigl(\kappa(\sigma)\bigr)=\sigma(\sigma+2).\]

Let $\varphi_{\kappa(\sigma)}$ be the positive first Dirichlet eigenfunction, normalized in $L^2(\Omega_{\kappa(\sigma)})$:
\[-\Delta_{\Sph^3}\varphi_{\kappa(\sigma)}=\mu_1\bigl(\kappa(\sigma)\bigr)\varphi_{\kappa(\sigma)}\quad\text{in }\Omega_{\kappa(\sigma)}, \qquad \varphi_{\kappa(\sigma)}=0\quad\text{on }\partial\Omega_{\kappa(\sigma)}.\]
The first eigenvalue is simple, and $\Omega_{\kappa(\sigma)}$ is invariant under $O(2)\times O(2)$.  Hence, $\varphi_{\kappa(\sigma)}$ is invariant under that group and depends only on the Hopf angle $\theta$.  By the choice of $\kappa(\sigma)$, this is the spherical eigenfunction corresponding to homogeneity $\sigma$ in $\R^4$.
Define
\[W(z):=|z|^\sigma\varphi_{\kappa(\sigma)}\left(\frac{z}{|z|}\right).\]
Then $W$ is $\sigma$-homogeneous and
\[\Delta W=0\quad\text{in }D_{\kappa(\sigma)}, \qquad W>0\quad\text{in }D_{\kappa(\sigma)}, \qquad W=0\quad\text{on }\partial D_{\kappa(\sigma)}.\]
For the remainder of the proof, we abbreviate $\kappa(\sigma)$ simply by $\kappa$.
\subsection{The transmission profile}\label{subsec:transmission-profile}

We now place two copies of $W$ in complementary cones. Let
\[\alpha=\alpha(\sigma):=\kappa(\sigma)^{-2}>1\]
and define
\[
A_+:=
\begin{pmatrix}
\alpha {\rm Id}_2&0\\
0&{\rm Id}_2
\end{pmatrix},
\qquad
A_-:=
\begin{pmatrix}
{\rm Id}_2&0\\
0&\alpha {\rm Id}_2
\end{pmatrix}.
\]
As $\sigma\downarrow0$, one has $\kappa(\sigma)\downarrow0$ and hence $\alpha(\sigma)\uparrow\infty$ (in particular, ellipticity degenerates as $\sigma\downarrow 0$). Let us define the linear elliptic operators with constant coefficients:
\[
\begin{aligned}
L_+ w&:={\rm tr}\,(A_+D^2w)={\rm div}(A_+\nabla w) = \alpha\Delta_x w+\Delta_y w,\\
L_- w&:={\rm tr}\,(A_-D^2w)={\rm div}(A_-\nabla w) = \Delta_x w+\alpha\Delta_y w.
\end{aligned}
\]
Consider the complementary cones
\[C_+:=\{|x|>|y|\}, \qquad C_-:=\{|y|>|x|\},\]
and their interface
\[\Gamma:=\{|x|=|y|\}.\]

Define the zero-extended functions
\[
U_+(x,y):=
\begin{cases}
W(\kappa x,y),&(x,y)\in C_+,\\
0,&(x,y)\notin C_+,
\end{cases}
\]
and
\[
U_-(x,y):=
\begin{cases}
W(\kappa y,x),&(x,y)\in C_-,\\
0,&(x,y)\notin C_-,
\end{cases}
\]
i.e., $U_-(x, y) = U_+(y, x)$. Both are nonnegative and $\sigma$-homogeneous.

\begin{lem}\label{lem:phase-equations}
The functions $U_+$ and $U_-$ satisfy
\[L_+U_+=0\quad\text{in }C_+, \qquad L_-U_-=0\quad\text{in }C_-.\]
\end{lem}

\begin{proof}
Inside $C_+$, write $X=\kappa x$ and $Y=y$.  Since $\alpha\kappa^2=1$,
\[L_+U_+=\alpha\Delta_xW(\kappa x,y)+\Delta_yW(\kappa x,y)=\alpha\kappa^2\Delta_XW+\Delta_YW=\Delta W=0.\]
We used here that if $(x, y) \in C_+$, then $(\kappa x, y)\in D_\kappa$. The same argument applies to $C_-$, since $(x, y)\in C_-$ implies $(\kappa y, x) \in D_\kappa$.
\end{proof}

The essential point is equality of the conormal source measures of the two zero extensions.  This is the same distributional compatibility that appears, in the Bellman-to-transmission direction, in \cite[Lemmas~2.1--2.2]{AnderssonMikayelyan2013} and \cite[Section~4]{CaffarelliDeSilvaSavin2018}.

\begin{prop}\label{prop:flux}
We have
\[L_+U_+=L_-U_- \qquad\text{distributionally in $\R^4$}.\]
\end{prop}

\begin{proof}
Notice from Lemma~\ref{lem:phase-equations} that the only possible contributions come from the boundary $\{|x| = |y|\}$. In fact, each $L_\pm U_\pm$ is a locally finite measure concentrated on $\{|x| = |y|\}$ and, since $U_\pm \ge 0$, has a positive density given by the corresponding conormal derivative on $\{|x| = |y|\}$. That is,
\[L_+ U_+ = |A_+ \nabla U_+ \cdot \nu_+|\mathcal{H}^{3}\res \{|x| = |y|\},\]
where $\nu_+$ is the outward unit normal to $C_+$ at $\Gamma$, given by $\nu_+ = \frac{1}{\sqrt{2}|x|}\left(-x, y\right)$.

Similarly, we have
\[L_- U_- = |A_- \nabla U_- \cdot \nu_-|\mathcal{H}^{3}\res \{|x| = |y|\},\]
with $\nu_- = -\nu_+$.

Now, by the separate rotational symmetry for $W$ in $x$ and $y$, there is a function $\Psi$ such that
\[U_+(x,y)=\Psi(r,s), \qquad r:=|x|,\quad s:=|y|,\]
in $C_+$ (and similarly, $U_-(x, y) = \Psi (s, r)$).  On the interface, for $t>0$, $\Psi(t,t)=0$.  Differentiating along the ray $t\mapsto(t,t)$ yields
\[\Psi_r(t,t)+\Psi_s(t,t)=0.\]
Let
\[d(t):=\Psi_r(t,t)=-\Psi_s(t,t),\]
so that
\[\nabla U_+=\left(d(t)\frac{x}{t},-d(t)\frac{y}{t}\right) \qquad\text{on }\Gamma\setminus\{0\}.\]
Consequently,
\[-A_+\nabla U_+\cdot\nu_+=-\left(\alpha d(t)\frac{x}{t},-d(t)\frac{y}{t}\right)\cdot\frac1{\sqrt2}\left(-\frac{x}{t},\frac{y}{t}\right)=\frac{\alpha+1}{\sqrt2}\,d(t).\]

Since $U_-(x,y)=\Psi(s,r)$ in $C_-$, one similarly has
\[-A_-\nabla U_-\cdot\nu_-=-\left(-d(t)\frac{x}{t},\alpha d(t)\frac{y}{t}\right)\cdot\frac1{\sqrt2}\left(\frac{x}{t},-\frac{y}{t}\right)=\frac{\alpha+1}{\sqrt2}\,d(t).\]
Thus the two conormal densities agree pointwise on $\Gamma\setminus\{0\}$, and since $U_\pm$ are $\sigma$-homogeneous, $U_\pm\in W^{1,2}_{\rm loc}(\R^4)$ and no point mass is created at the origin. This completes the proof.
\end{proof}

\subsection{Proof of the main result}\label{subsec:proof-main}

Define
\[q:=U_+-U_-.\]
Since the supports of $U_+$ and $U_-$ occupy complementary cones,
\[q^+=U_+, \qquad q^-=U_-,\]
where $q^+=\max\{q,0\}$ and $q^-=\max\{-q,0\}$.  Proposition \ref{prop:flux} becomes
\begin{equation}\label{eq:compatibility}L_+q^+=L_-q^- \qquad\text{distributionally in $\R^4$}.\end{equation}
Equivalently, we have a broken quasilinear PDE
\[\diver\bigl(A_+\nabla q^+-A_-\nabla q^-\bigr)=0.\]

\begin{cor}\label{cor:broken-conductivity}
For every $\sigma\in(0,1)$, there exist constant uniformly elliptic, nonproportional matrices $A_+,A_-\in\Sym^4$ and a nonzero function
\[q\in C^\sigma_{\mathrm{loc}}(\R^4)\cap W^{1,2}_{\mathrm{loc}}(\R^4),\]
homogeneous of degree $\sigma$, such that
\[\diver\bigl(A_+\nabla q^+-A_-\nabla q^-\bigr)=0 \qquad\text{in }\R^4.\]
Moreover,
\[\{q=0\}=\Gamma=\{|x|=|y|\}, \qquad q\in C^1(\R^4\setminus\{0\}),\]
and $q$ is not $C^\beta$ at the origin for any $\beta>\sigma$.
\end{cor}

\begin{proof}
Take the matrices $A_+,A_-$ and the function $q=U_+-U_-$ constructed above with $q>0$ in $C_+$, $q<0$ in $C_-$, and $q=0$ on $\Gamma$. The regularity of the angular profiles up to the boundary, together with homogeneity and zero extension, gives $q\in C^\sigma_{\mathrm{loc}}(\R^4)\cap W^{1,2}_{\mathrm{loc}}(\R^4)$. The matching of normal derivatives in the proof of Proposition~\ref{prop:flux} gives the $C^1$ regularity outside of the origin.
\end{proof}

We can now prove the main theorem.

\begin{proof}[Proof of Theorem \ref{thm:main}]
The zero extensions $U_\pm$ have Lipschitz angular profiles, since the spherical eigenfunction is smooth up to the boundary and vanishes there.  The homogeneous-extension therefore implies that $q^+=U_+$ and $q^-=U_-$ belong to $C^\sigma_{\rm loc}(\R^4)$, including at the origin. Apply Lemma \ref{lem:lifting} with
\[L_1=L_-, \qquad L_2=L_+,\qquad q_1 = q^+,\qquad q_2 = q^-,\]
since \eqref{eq:compatibility} holds.  We obtain a $(2+\sigma)$-homogeneous function
$v\in C^{2,\sigma}_{\mathrm{loc}}(\R^4)$
such that
\[L_-v=q^+, \qquad L_+v=q^-.\]
Since $q^+,q^-\geq0$ and $q^+q^-\equiv0$,
\[\min\{L_-v,L_+v\}=0 \qquad\text{in }\R^4,\]
so that letting $u = -v$ gives the desired result with $A_1 = A_-$ and $A_2 = A_+$.

Since $\sigma\in(0,1)$ was arbitrary, the construction applies to every prescribed exponent in this interval.
\end{proof}

\section{Proof of Theorem~\ref{thm:three-operators}}\label{sec:three-operators}

The proof is an explicit construction based on three rotated sectors.

\begin{proof}[Proof of Theorem~\ref{thm:three-operators}]
We use the identity
\[\arctan 2+\arctan 3=\frac{3\pi}{4}.\]
Let $R_\theta$ denote counterclockwise rotation through the angle $\theta$, and define
\[
A_0:=
\begin{pmatrix}
3&0\\
0&1
\end{pmatrix},
\qquad
A_j:=R_{2j\pi/3}A_0R_{2j\pi/3}^{T},
\qquad j=1,2,
\]
so that each matrix has eigenvalues 1 and 3. Explicitly,
\begin{equation}\label{eq:three-matrices-explicit}
A_1=
\frac12 \begin{pmatrix}
3&-\sqrt3\\[1mm]
-\sqrt3&5
\end{pmatrix},
\qquad
A_2=
\frac12 \begin{pmatrix}
3&\sqrt3\\[1mm]
\sqrt3&5
\end{pmatrix}.
\end{equation}

Consider
\[\Gamma_0:=\bigl\{(r\cos\theta,r\sin\theta):r>0,\ |\theta|\leq\pi/3\bigr\}.\]
For $x=(x_1,x_2)\in\Gamma_0$, set
\[z:=\frac{x_1}{\sqrt3}+ix_2, \qquad u_0(x):={\rm Re}\,(z^{2+\sigma_*}).\]
Notice $\Gamma_0\subset \{x_1 > 0\}$, so the branch $z^{2+\sigma_*}$ is well defined.  With $X=x_1/\sqrt3$ and $Y=x_2$, the function ${\rm Re}\,(X+iY)^{2+\sigma_*}$ is harmonic.  Hence
\begin{equation}\label{eq:u0-active}{\rm tr}\,(A_0D^2u_0)=3(u_0)_{11}+(u_0)_{22}=0 \qquad\text{in }\Gamma_0.\end{equation}

We next check that $u_0$ has zero Euclidean normal derivative on both boundary rays.  On the upper ray $\theta=\pi/3$, write $z=\rho e^{i\arctan 3}$ and let
\[\nu=(-\sin(\pi/3),\cos(\pi/3))=(-\sqrt{3}/2,1/2),\]
be the outward unit normal to $\Gamma_0$.  Direct differentiation gives
\begin{align*}
\partial_{x_1}u_0
&=\frac{2+\sigma_*}{\sqrt3}\rho^{1+\sigma_*}
  \cos\bigl((1+\sigma_*)\arctan 3\bigr),\\
\partial_{x_2}u_0
&=-(2+\sigma_*)\rho^{1+\sigma_*}
  \sin\bigl((1+\sigma_*)\arctan 3\bigr).
\end{align*}
We obtain, using $(1+\sigma_*)\arctan 3=\arctan 3+\arctan 2=\frac{3\pi}{4}$,
\begin{align*}
\partial_\nu u_0
&=-\frac{2+\sigma_*}{2}\rho^{1+\sigma_*}
\left[
\cos\bigl((1+\sigma_*)\arctan 3\bigr)
+
\sin\bigl((1+\sigma_*)\arctan 3\bigr)
\right]
=0.
\end{align*}
The same conclusion holds on the lower ray by symmetry.

Let
\[\Gamma_j:=R_{2j\pi/3}\Gamma_0, \qquad j=0,1,2.\]
These sectors cover $\R^2\setminus\{0\}$ and have disjoint interiors.  Define
\[u(x):=u_0(R_{-2j\pi/3}x) \qquad\text{for }x\in\Gamma_j.\]
The matching argument below shows that this definition is independent of the choice of $j$ on the common boundary rays.

To see the matching across the boundary rays, write
\[u_0(r,\theta)=r^{2+\sigma_*}\phi(\theta), \qquad |\theta|\leq\pi/3.\]
The function $\phi$ is smooth and even, with $\phi'(\pm\pi/3)=0$. Evenness matches $\phi$ and $\phi''$ at the endpoints, while the vanishing normal derivative matches $\phi'$.  Its extension with period $2\pi/3$ is therefore $C^{2,1}$.  In particular, by $(2+\sigma_*)$-homogeneity we have $u\in C^{2,\sigma_*}_{\mathrm{loc}}(\R^2)$.

It remains to check the equation in the interior of $\Gamma_0$.  Write
\[z=\rho e^{i\Theta}, \qquad |\Theta|<\arctan 3, \qquad \psi:=\sigma_*\Theta.\]
For $f(z)=z^{2+\sigma_*}$ we have
$f''(z)=(2+\sigma_*)(1+\sigma_*)\rho^{\sigma_*}e^{i\psi}$.
Taking its real and imaginary parts gives the Hessian in the $(X,Y)$ variables; the change $X=x_1/\sqrt3$, $Y=x_2$ contributes the factors $1/3$ and $1/\sqrt3$ in the first and mixed entries, respectively.  Thus
\[D^2u_0=(2+\sigma_*)(1+\sigma_*)\rho^{\sigma_*}
\begin{pmatrix}
\frac13\cos\psi&-\frac1{\sqrt3}\sin\psi\\[1mm]
-\frac1{\sqrt3}\sin\psi&-\cos\psi
\end{pmatrix}.\]
Using \eqref{eq:three-matrices-explicit}, we find
\begin{align}
{\rm tr}\,(A_1D^2u_0)
&=(2+\sigma_*)(1+\sigma_*)\rho^{\sigma_*}
  \bigl(-2\cos\psi+\sin\psi\bigr),
\label{eq:three-trace-one}\\
{\rm tr}\,(A_2D^2u_0)
&=(2+\sigma_*)(1+\sigma_*)\rho^{\sigma_*}
  \bigl(-2\cos\psi-\sin\psi\bigr).
\label{eq:three-trace-two}
\end{align}
Now $|\psi|\leq\sigma_*\arctan 3=\arctan 2$.  Hence
\[|\sin\psi|\leq2\cos\psi,\]
so both quantities in \eqref{eq:three-trace-one}--\eqref{eq:three-trace-two} are nonpositive.  Together with \eqref{eq:u0-active}, this proves
\[\max_{j=0,1,2}{\rm tr}\,(A_jD^2u_0)=0 \qquad\text{in }\Gamma_0.\]
Rotation gives the same identity in the other sectors.  Since the Hessian is continuous across the boundary rays and vanishes at the origin, the equation holds classically in all of $\R^2$.  Finally, $u(1,0)=3^{-(2+\sigma_*)/2}>0$, so $u$ is nonzero.  Along the positive $x_1$-axis the Hessian formula above is a nonzero constant matrix times $r^{\sigma_*}$; since $\sigma_*<1$, $D^2u$ is not Lipschitz at the origin.

\end{proof}

\end{document}